\documentclass{article}

\usepackage{latexsym}
\usepackage{graphicx}
\usepackage{enumitem}
\usepackage{amsmath}
\usepackage{amssymb}
\usepackage{amsthm}
\usepackage{multirow}
\usepackage{color}
\usepackage{url}
\def\zhou#1 {\fbox {\footnote {\ }}\ \footnotetext { From Yue: {\color{red}#1}}}

\newtheorem{theorem}{Theorem}[section]

\newtheorem{remark}[theorem]{Remark}
\newtheorem{lemma}[theorem]{Lemma}

\newtheorem{corollary}[theorem]{Corollary}

\newtheorem{problem}{Problem}

\newcommand{\Gal}{\mathrm{Gal}}

\newcommand{\GL}{\mathrm{GL}}

\newcommand{\F}{\mathbb F}

\newcommand{\RN}[1]{%
	\textup{\uppercase\expandafter{\romannumeral#1}}%
}

\title{New lower bounds for partial $k$-parallelisms}
\author{Tao Zhang$^{\text{a}}$ and Yue Zhou$^{\text{b,}}$\thanks{Corresponding author (e-mail: yue.zhou.ovgu@gmail.com).} \\
\footnotesize $^{\text{a}}$ School of Mathematics and Information Science, Guangzhou University, Guangzhou 510006, China.\\
\footnotesize $^{\text{b}}$ Department of Mathematics, National University of Defense Technology, Changsha 410073, China.}
\begin{document}
\date{}\maketitle
\begin{abstract}
Due to the applications in network coding, subspace codes and designs have received many attentions. Suppose that $k\mid n$ and $V(n,q)$ is an $n$-dimensional space over the finite field $\mathbb{F}_{q}$. A $k$-spread is a $\frac{q^n-1}{q^k-1}$-set of $k$-dimensional subspaces of $V(n,q)$ such that each nonzero vector is covered exactly once. A partial $k$-parallelism in $V(n,q)$ is a set of pairwise disjoint $k$-spreads. As the number of $k$-dimensional subspaces in $V(n,q)$ is ${n \brack k}_{q}$, there are at most ${n-1 \brack k-1}_{q}$ spreads in a partial $k$-parallelism.

By studying the independence numbers of Cayley graphs associated to a special type of partial $k$-parallelisms in $V(n,q)$, we obtain new lower bounds for partial $k$-parallelisms. In particular, we show that there exist at least $\frac{q^{k}-1}{q^{n}-1}{n-1 \brack k-1}_q$ pairwise disjoint $k$-spreads in $V(n,q)$.

\smallskip

\end{abstract}

\section{Introduction}
In 2008, K\"otter and Kschischang \cite{KK08} found an important application of subspace codes in random network coding. After their work, the theory of subspace codes and designs ($q$-analogs of codes and designs) has developed rapidly, see \cite{BPV13,ES09,EV11,SE11,SE2011,SK09,SKK08,S10} and the references therein. There are a lot of $q$-analogs of combinatorial structures been studied in recent years. For example, constant dimension codes \cite{ES09}, $q$-Steiner systems \cite{BEOVW16}, $t$-designs over finite fields \cite{FLV14}, large sets of subspace designs \cite{BKKL17}. For a survey on subspace codes and designs, we refer the reader to \cite{ES16}.

Let $n$ and $k$ be positive integers with $k<n$, and let $q$ be a prime power. Let $V(n,q)$ denote an $n$-dimensional space over the finite field $\mathbb{F}_{q}$ and $\mathcal{G}_{q}(n,k)$ denote the set of all $k$-dimensional subspaces of $V(n,q)$. It is well known that
\begin{align*}
|\mathcal{G}_{q}(n,k)|={n \brack k}_q:=\frac{(q^{n}-1)(q^{n-1}-1)\cdots(q^{n-k+1}-1)}{(q^{k}-1)(q^{k-1}-1)\cdots(q-1)},
\end{align*}
 where ${n \brack k}_{q}$ is the $q$-ary Gaussian coefficient.

A \emph{partial $k$-spread} $S$ of $V(n,q)$ is a collection of $k$-dimensional subspaces of $V(n,q)$ such that nonzero vectors are covered at most once. It is easy to see that a partial $k$-spread of $V(n,q)$ is a constant dimension code in $\mathcal{G}_{q}(n,k)$ with minimum distance $k$. Given positive integers $n$ and $k$ with $k<n$, it is natural to ask the problem of finding the maximum size of a partial $k$-spread of $V(n,q)$. This problem has been extensively studied, for a recent results on partial $k$-spread, see \cite{NS17}. If $S$ contains all the nonzero vectors of $V(n,q)$, then it is called a \emph{$k$-spread}. In 1954, Andr\'e \cite{A54} proved that a $k$-spread of $V(n,q)$ exists if and only if $k$ divides $n$. The size of a $k$-spread in $V(n,q)$ is $\frac{q^{n}-1}{q^{k}-1}$.

Assume that $k$ divides $n$. A \emph{partial $k$-parallelism} $P$ of $V(n,q)$ is a collection of $k$-spreads of $V(n,q)$ that are pairwise disjoint. If $P$ contains all the $k$-dimensional subspaces of $V(n,q)$, then it is called a $k$-parallelism. Let $P(n,k,q)$ denote the maximum size of a partial $k$-parallelism of $V(n,q)$. Note that the size of a $k$-spread in $V(n,q)$ is $\frac{q^{n}-1}{q^{k}-1}$ and the number of $k$-dimensional subspaces in $V(n,q)$ is ${n \brack k}_{q}$. Hence $P(n,k,q)\le {n-1 \brack k-1}_{q}$, and $P(n,k,q)={n-1 \brack k-1}_{q}$
if and only if there exists a $k$-parallelism in $V(n,q)$.

Some $2$-parallelisms of $V(n,q)$ are known for many years. For $n$ even, a $2$-parallelism in $V(n,2)$ was found in the context of Preparata codes \cite{B76,BVW83}. In \cite{B74} Beutelspacher showed that there exists a 2-parallelism in $V(2^{i},q)$ for $i\ge2$ and $q$ prime power. Recently, two more sporadic examples of $k$-parallelisms were found: one $2$-parallelism in $V(6,3)$ \cite{EV12} and one $3$-parallelism in $V(6,2)$ \cite{S02}.

As there are only a few constructions of $k$-parallelisms, it is naturally to state the following problem.
\begin{problem}
Given positive integers $n,k$ with $k|n$ and a prime power $q$, find the largest possible size $P(n,k,q)$ of partial $k$-parallelism of $V(n,q)$.
\end{problem}

In \cite{B90}, Beutelspacher proved that $P(n,2,q)\ge q^{2\lfloor\text{log}(n-1)\rfloor}+\cdots+q+1$ for $n$ even and $q$ a prime power. Recently, Etzion \cite{E15} showed that $P(n,k,2)\ge2^{k}-1$ for $k|n$ and $n>k$, and $P(n,k,q)\ge2$ for $k|n$, $n>k$ and $q$ a prime power. In the same paper it was proven that $P((m+1)k,k,q)\ge P(mk,k,q)$ for $m\ge 2$ by a recursive construction. Below we improve the previous results by showing that $P(n,k,q)\ge\frac{q^{k}-1}{q^{n}-1}{n-1 \brack k-1}_{q}$.

The rest of the paper are organized as follows. In Section \ref{preliminaries}, we introduce some basics of graphs and general linear groups. In Section \ref{partial k-parallelism}, we give a new lower bound for $P(n,k,q)$. As a consequence, we can show that there exist at least $\frac{q^{k}-1}{q^{n}-1}{n-1 \brack k-1}_{q}$ pairwise disjoint $k$-spreads in $V(n,q)$.

\section{Preliminaries}\label{preliminaries}
\subsection{Graph theory}
A graph $G$ consists of a set of vertices $V(G)$ and a set of edges $E(G)$. Two vertices $u$ and $v$ are called adjacent if $\{u,v\}\in E(G)$.

Let $H$ be a group and $S$ be a subset of $H$ such that $S^{-1}=S$ and $e\not\in S$. Here $S^{-1}=\{s^{-1}: s\in S\}$. The \emph{Cayley graph} $\Gamma(H,S)$ is a graph with vertex set $H$ in which two distinct vertices $g,h$ are adjacent if and only if $g^{-1}h\in S$. Here $S$ is called the \emph{generating set}.

An \emph{independent set} of a graph $G$ is a subset of $V(G)$ such that every pair of vertices are not adjacent. The \emph{independence number} $\alpha(G)$ of a graph $G$ is the cardinality of the largest independent set. Formally,
 \begin{align*}
 \alpha(G)=\text{max}\{|U|: U\subseteq V(G),U\text{ is an independent set}\}.
 \end{align*}

 The independence number is an important parameter in graph theory and has been studied for a long time. It is also related to some other parameters such as chromatic number, clique number, and so on. We will use the following result on the independence number proved by Caro and Wei \cite{AS16}:
  
Let $d_{v}$ denote the degree of a vertex $v$.
\begin{lemma}\rm{\cite{AS16}}\label{lemma1}
$\alpha(G)\ge\sum_{v\in V}\frac{1}{d_{v}+1}.$
\end{lemma}

By Lemma \ref{lemma1}, for a Cayley graph $\Gamma$ with vertex set $H$ and generating set $S$, we have $\alpha(\Gamma)\ge\frac{|H|}{|S|+1}$.

\subsection{General linear group}
The general linear group $\GL(n,q)$ is the group of non-singular linear transformations of $V(n,q)$. It is isomorphic to the multiplicative group of $n\times n$ non-singular matrices whose entries come from $\mathbb{F}_{q}$. It is well known that
\begin{align*}
|\GL(n,q)|=(q^{n}-1)(q^{n}-q)\dots(q^{n}-q^{n-1}).
\end{align*}

\begin{lemma}
For any $M\in\GL(n,q)$ and any $k$-dimensional subspace $U$ of $V(n,q)$, $\text{dim}(M(U))=k$.
\end{lemma}
\begin{lemma}\label{lemma3}
Let $k,n$ be integers such that $k|n$, and $M\in \GL(n,q)$. If $\mathcal{S}$ is a $k$-spread of $V(n,q)$, then $\{M(U): U\in\mathcal{S}\}$ is also a $k$-spread of $V(n,q)$.
\end{lemma}
\begin{proof}
Suppose $v\in M(U)\cap M(V)$, where $U,V\in\mathcal{S}$. Since $M$ is non-singular, then $M^{-1}(v)\in U\cap V$. Hence $M^{-1}(v)=0$, and then $v=0$. Therefore $\{M(U): U\in\mathcal{S}\}$ is a $k$-spread of $V(n,q)$.
\end{proof}

\begin{lemma}\label{lemma2}
For any $M\in\GL(n,q)$ and any two $k$-dimensional subspaces $U,V$ of $V(n,q)$,
\begin{align*}
|\{M: M\in \GL(n,q), M(U)=V\}|=q^{k(n-k)}|\GL(k,q)||\GL(n-k,q)|.
\end{align*}
\end{lemma}
\begin{proof}
Let $S_{0}=\{M: M\in \GL(n,q), M(U)=U\}$ and $S_{1}=\{M: M\in \GL(n,q), M(U)=V\}$. Fix a matrix $M_{0}\in S_{1}$. Then for any $M\in S_{1}$, we have $M_{0}^{-1}M\in S_{0}$. On the other hand, for any $N\in S_{0}$, we have $M_{0}N\in S_{1}$. Hence $|S_{0}|=|S_{1}|$.

Any basis $u_{1},u_{2},\dots,u_{k}$ of $U$ can be extended to a basis over $\mathbb{F}_{q}$. Let $u_{1},u_{2},\cdots,u_{k},v_{1},\cdots,v_{n-k}$ be a basis of $V(n,q)$ over $\mathbb{F}_{q}$. For any $M\in \GL(n,q)$, $M(U)=U$ if and only if
\begin{align*}
M\left(
   \begin{array}{c}
     a_{1} \\
     \vdots \\
     a_{k} \\
     0 \\
     \vdots \\
     0 \\
   \end{array}
 \right)=\left(
   \begin{array}{c}
     b_{1} \\
     \vdots \\
     b_{k} \\
     0 \\
     \vdots \\
     0 \\
   \end{array}
 \right)
\end{align*}
for some $a_{i},b_{i}\in\mathbb{F}_{q}$, $i=1,\cdots,k$. Thus $M$ must be of the form
\begin{align*}
\left(
  \begin{array}{cc}
    A & C \\
    0 & B \\
  \end{array}
\right),
\end{align*}
where $A\in \GL(k,q)$, $B\in \GL(n-k,q)$ and $C$ is any $k\times(n-k)$ matrix. Hence $|\{M: M\in \GL(n,q), M(U)=V\}|=|S_{0}|=q^{k(n-k)}|\GL(k,q)||\GL(n-k,q)|$.
\end{proof}

\section{A lower bound for partial $k$-parallelisms}\label{partial k-parallelism}
In this section, we give a new lower bound for partial $k$-parallelisms. Let $k,n$ be positive integers such that $k|n$. Let $q$ be a prime power, $N=\frac{q^{n}-1}{q^{k}-1}$ and let $\omega$ be a primitive element of $\mathbb{F}_{q^{n}}$. Let $V_{i}=\omega^{i}\mathbb{F}_{q^{k}}, i=0,1,\dots,N-1$. Then $\{V_{i}: i=0,1,\cdots,N-1\}$ forms a $k$-spread.

Let
\begin{align}\label{eq3}
S_{ij}=\{M\in \GL(n,q): M(V_{i})=V_{j}\}.
\end{align}
%
\begin{theorem}\label{mainthm}
Let $q$ be a prime power and $k,n$ be positive integers such that $k|n$. Then
\begin{align*}
	P(n,k,q)\ge\frac{|\GL(n,q)|}{|\cup_{i,j=0}^{N-1}S_{ij}|},
\end{align*}
where $S_{ij}$ is defined by Equation~(\ref{eq3}).
\end{theorem}
\begin{proof}
	Note that $\{V_{i}: i=0,1,\cdots,N-1\}$ forms a $k$-spread. By Lemma~\ref{lemma3}, for any $M\in \GL(n,q)$, $\{M(V_{i}): i=0,1,\cdots,N-1\}$ also forms a $k$-spread.

	Now we define a graph $G$ with vertex set $V(G)=\GL(n,q)$ in which two vertices $M_{1},M_{2}\in \GL(n,q)$ are adjacent if and only if $M_{1}^{-1}M_{2}\in S$, where
	$$S=\{M: M\in \GL(n,q), M\ne I, M(V_{i})=V_{j}\text{ for some }0\le i,j\le N-1\}.$$
	
	Note that if $M(V_{i})=V_{j}\text{ for some }0\le i,j\le N-1$, then $M^{-1}(V_{j})=V_{i}$ and so $M^{-1}\in S$. Hence $S=S^{-1}$. Thus the graph $G$ is a Cayley graph. The degree of graph $G$ is $|S|$.

	If $\mathcal{I}$ is an independent set of the graph $G$, we claim that $\{M(V_{i}): i=0,1,\cdots,N-1\}$, $M\in\mathcal{I}$ form a partial $k$-parallelism. Otherwise, there exist $M_{1}\ne M_{2}\in\mathcal{I}$ and $0\le i,j\le N-1$ such that $M_{1}(V_{i})=M_{2}(V_{j})$. Then $M_{2}^{-1}M_{1}(V_{i})=V_{j}$, so $M_{2}^{-1}M_{1}\in S$. Hence $M_{1}$ and $M_{2}$ are adjacent, which contradicts the fact that $\mathcal{I}$ is an independent set of graph $G$. Therefore $\{M(V_{i}): i=0,1,\cdots,N-1\}$, $M\in\mathcal{I}$ form a partial $k$-parallelism. So we have $P(n,k,q)\ge\alpha(G)$.

As $|S|=|\cup_{i,j=0}^{N-1}S_{i,j}|-1$ we deduce from Lemma 2.1 that
	\[P(n,k,q)\ge\alpha(G)\ge\frac{|V(G)|}{|S|+1}\ge\frac{|\GL(n,q)|}{|\cup_{i,j=0}^{N-1}S_{ij}|}. \qedhere \]
\end{proof}

To get the precise lower bound in Theorem~\ref{mainthm}, one needs to compute the size of $\cup_{i,j=1}^{N}S_{ij}$. Next let us look at the case $n=2k$.
\begin{theorem}\label{thm2}
Let $q$ be a prime power and $k$ be a positive integer, then
\begin{align*}
P(2k,k,q)\ge\frac{|\GL(2k,q)|}{L},
 \end{align*}
 where $L=N^{2}q^{k^{2}}|\GL(k,q)|^{2}-\frac{1}{2}(N(N-1))^{2}|\GL(k,q)|^{2}+\frac{1}{3!}(N(N-1)(N-2))^{2}|\GL(k,q)|+\sum_{i=4}^{N}(-1)^{i+1} \frac{1}{i!}(N(N-1)(N-2))^{2}\sum_{l|k}\left(\left|\GL(\frac{k}{l},q^{l})\right|-\left| \bigcup_{l\mid m, l\neq m} \GL(\frac{k}{m}, q^m)\rtimes \mathrm{Gal}(\F_{q^m}/\F_q)\right|\right)l[q^{l}-2]_{i-3}$ and $[n]_{j}=n(n-1)\cdots(n-j+1)$.
\end{theorem}
\begin{proof}
By Theorem \ref{mainthm}, we only need to compute the size of $\cup_{i,j=0}^{N-1}S_{ij}$. In the following, we determine the size of $\cup_{i,j=0}^{N-1}S_{ij}$ by the inclusion-exclusion principle
\begin{equation}\label{eq2}
|\cup_{i,j=0}^{N-1}S_{ij}|=\sum_{i,j=0}^{N-1}|S_{ij}|-\frac{1}{2!}\sum_{\substack{i_{0}\ne i_{1}\\   j_{0}\ne j_{1}}}|S_{i_{0}j_{0}}\cap S_{i_{1}j_{1}}|+\frac{1}{3!}\sum_{\substack{i_{0}, i_{1}, i_{2}\text{ are pairwise distinct}\\   j_{0}, j_{1}, j_{2}\text{ are pairwise distinct}}}|S_{i_{0}j_{0}}\cap S_{i_{1}j_{1}}\cap S_{i_{2}j_{2}}|-\cdots.
\end{equation}

First, by Lemma~\ref{lemma2}, we have $\sum_{i,j=0}^{N-1}|S_{ij}|=N^{2}q^{k^{2}}|\GL(k,q)|^{2}$.

Next we consider $S_{i_{0}j_{0}}\cap S_{i_{1}j_{1}}$, where $i_{0}\ne i_{1},  j_{0}\ne j_{1}$. By change of basis, we may assume that $i_{0}=j_{0}=0$ and $i_{1}=j_{1}=1$. If $M\in S_{i_{0}j_{0}}\cap S_{i_{1}j_{1}}$, then it must be of the form
$\left(
   \begin{array}{cc}
     A & 0 \\
     0 & B \\
   \end{array}
 \right),
$
where $A,B\in \GL(k,q)$. Hence $|S_{i_{0}j_{0}}\cap S_{i_{1}j_{1}}|=|\GL(k,q)|^{2}$. Therefore
\begin{align*}
-\frac{1}{2!}\sum_{\substack{i_{0}\ne i_{1}\\   j_{0}\ne j_{1}}}|S_{i_{0}j_{0}}\cap S_{i_{1}j_{1}}|=-\frac{1}{2}(N(N-1))^{2}|\GL(k,q)|^{2}.
\end{align*}

Now we consider $S_{i_{0}j_{0}}\cap S_{i_{1}j_{1}}\cap S_{i_{2}j_{2}}$, where $i_{0}, i_{1}, i_{2}\text{ are pairwise distinct and } j_{0}, j_{1}, j_{2}\text{ are pairwise distinct}$. By change of basis, we may assume that $i_{0}=j_{0}=0$ and $i_{1}=j_{1}=1$. Moreover, there exists $a_{i},b_{i}\in\mathbb{F}_{q^{k}}^{*}$ with $i=0,1$ such that
	\begin{align*}
	\omega^{i_{2}}&=a_{0}+a_{1}\omega,\\
	\omega^{j_{2}}&=b_{0}+b_{1}\omega.
	\end{align*}
	
	For any $\varphi\in S_{00}\cap S_{11}\cap S_{i_{2}j_{2}}\subseteq \GL(2k,q)$ (here we consider the elements of $\GL(2k,q)$ as the linear maps from $\mathbb{F}_{q^{2k}}$ to $\mathbb{F}_{q^{2k}}$), we define $\varphi_i :\F_{q^k} \rightarrow \F_{q^k}$ by
    \begin{align*}
		\varphi(\omega^{i}x)=\omega^{i}\varphi_{i}(x), x\in\mathbb{F}_{q^{k}}, i=0,1.
	\end{align*}
	As $\varphi$ maps $V_i=\{\omega^i x : x\in \F_{q^k}\}$ to itself for $i=0,1$, $\varphi_0$ and $\varphi_1$ must be in $\GL(k,q)$.
	By calculation,
	\begin{equation}\label{eq:phi_1}
		\varphi(\omega^{i_{2}}x)=\varphi((a_{0}+a_{1}\omega)x)=\varphi_{0}(a_{0}x)+\omega\varphi_{1}(a_{1}x),
	\end{equation}
	for each $x\in \F_{q^k}$. As $\varphi(\omega^{i_{2}}x)\in V_{j_2}$,
	\begin{equation}\label{eq:phi_2}
		\varphi(\omega^{i_{2}}x) = \omega^{j_2} x' =(b_{0}+b_{1}\omega)x'=b_{0}x'+\omega b_{1}x',
	\end{equation}
	for some $x' \in \F_{q^k}$.

Note that $a_{i},b_{i}\in\mathbb{F}_{q^{k}}^{*}$, by Equations \eqref{eq:phi_1} and \eqref{eq:phi_2}, $\varphi_{1}(x)=\frac{b_{1}}{b_{0}}\varphi_{0}(\frac{a_{0}}{a_{1}}x)$, which means that $\varphi_0$ completely determines $\varphi_1$. Conversely, all $\varphi_i$'s defined in this way give rise to a map $\varphi$ from $V_{i_2}$ to $V_{j_2}$. Thus
\begin{align*}
\frac{1}{3!}\sum_{\substack{i_{0}, i_{1}, i_{2}\text{ are pairwise distinct}\\   j_{0}, j_{1}, j_{2}\text{ are pairwise distinct}}}|S_{i_{0}j_{0}}\cap S_{i_{1}j_{1}}\cap S_{i_{2}j_{2}}|=\frac{1}{3!}(N(N-1)(N-2))^{2}|\GL(k,q)|.
\end{align*}

For the fourth term of Equation~(\ref{eq2}), that is $S_{i_{0}j_{0}}\cap S_{i_{1}j_{1}}\cap S_{i_{2}j_{2}}\cap S_{i_{3}j_{3}}$, where $i_{0}, i_{1}, i_{2},i_{3}$ are pairwise distinct and $j_{0}, j_{1}, j_{2},j_{3}$ are pairwise distinct. Without loss of generality, we may assume that $i_{0}=j_{0}=0$, $i_{1}=j_{1}=1$ and $\omega^{i_{2}}=\omega^{j_{2}}=1+\omega$. Moreover, there exists $a_{i},b_{i}\in\mathbb{F}_{q^{k}}^{*}$ with $i=0,1$ such that
	\begin{align}
	\label{eq:omega_is}	\omega^{i_{3}}&=a_{0}+a_{1}\omega,\\
	\label{eq:omega_js}	\omega^{j_{3}}&=b_{0}+b_{1}\omega,
	\end{align}
and $\frac{a_{1}}{a_{0}},\frac{b_{1}}{b_{0}}\ne1$.

Let $\varphi\in S_{i_{0}j_{0}}\cap S_{i_{1}j_{1}}\cap S_{i_{2}j_{2}}\cap S_{i_{3}j_{3}}$. By a similar discussion as for Equations~\eqref{eq:phi_1} and \eqref{eq:phi_2} , we have
\begin{align}
\varphi_{1}(x)=\varphi_{0}(x),\\
\varphi_{1}(x)=u^{-1}\varphi_{0}(vx),
\end{align}
where $u=\frac{b_{0}}{b_{1}}$, $v=\frac{a_{0}}{a_{1}}$. Then we have
\begin{align}
u\varphi_{0}(x)=\varphi_{0}(vx)\text{ for all }x\in\mathbb{F}_{q^{k}}.
\end{align}

This condition gives a strong restriction on $u$ and $v$. Suppose that $l$ is a divisor of $k$. By looking at the $q$-polynomial associated with $\varphi_{0}$, it is not difficult to see that if $\varphi_{0}\in\GL(\frac{k}{l},q^{l})\rtimes \Gal(\F_{q^l}/\F_q)$, then we may choose $v\in\mathbb{F}_{q^{l}}\setminus\{0,1\}$ and $u=v^{q^{r}}$, which means $\varphi_{0}(x)=x^{q^{r}}\circ \psi_{0}(x)$ for some $\psi_{0}\in \GL(\frac{k}{l},q^{l})$. Hence
\begin{align*}
&-\frac{1}{4!}\sum_{\substack{i_{0}, i_{1}, i_{2},i_{3}\text{ are pairwise distinct}\\   j_{0}, j_{1}, j_{2},j_{3}\text{ are pairwise distinct}}}|S_{i_{0}j_{0}}\cap S_{i_{1}j_{1}}\cap S_{i_{2}j_{2}}\cap S_{i_{3}j_{3}}|\\
=&-\frac{1}{4!}(N(N-1)(N-2))^{2}\sum_{l|k}\left(\left|\GL(k/l,q^{l})\right|-\left| \bigcup_{l\mid m, l\neq m} \GL(k/m, q^m)\rtimes \mathrm{Gal}(\F_{q^m}/\F_q)\right|\right)l(q^{l}-2).
\end{align*}

To extend the previous calculation to the rest terms of Equation~(\ref{eq2}), we have
\begin{align*}
|\cup_{i,j=1}^{N}S_{ij}|=&N^{2}q^{k^{2}}|\GL(k,q)|^{2}-\frac{1}{2}(N(N-1))^{2}|\GL(k,q)|^{2}+\frac{1}{3!}(N(N-1)(N-2))^{2}|\GL(k,q)|+\\
&\sum_{i=4}^{N}\frac{(-1)^{i+1}}{i!}(N(N-1)(N-2))^{2}\sum_{l|k}\left(\left|\GL(k/l,q^{l})\right|-\left| \bigcup_{l\mid m, l\neq m} \GL(k/m, q^m)\rtimes \mathrm{Gal}(\F_{q^m}/\F_q)\right|\right)l[q^{l}-2]_{i-3},
\end{align*}
where $[n]_{j}=n(n-1)\cdots(n-j+1)$.
\end{proof}

For $n\ge 3k$, it appears more difficult to give an exact formula for $|\cup_{i,j=1}^{N}S_{ij}|$. Hence we give the following lower bound for the general case.
\begin{theorem}\label{thm3}
Let $q$ be a prime power and assume that $k|n$. Then,
\begin{align*}
P(n,k,q)\ge\frac{|\GL(n,q)|}{(\frac{q^{n}-1}{q^{k}-1})^{2}q^{k(n-k)}|\GL(k,q)||\GL(n-k,q)|-\frac{(q^{n}-1)^{2}}{q^{k}-1}+(q^{n}-1)}.
\end{align*}
\end{theorem}
\begin{proof}
By Lemma~\ref{lemma2}, we have
\begin{align}
|\cup_{i,j=0}^{N-1}S_{ij}|\le\cup_{i,j=0}^{N-1}|S_{ij}|=(\frac{q^{n}-1}{q^{k}-1})^{2}q^{k(n-k)}|\GL(k,q)||\GL(n-k,q)|.\label{eq1}
\end{align}

Let $M_{i}\in \GL(n,q)$ such that $M_{i}(x)=\omega^{i}x$ for $i=0,1,\dots,q^{n}-2$. Then $M_{i}(V_{j})=V_{(i+j)\pmod{\frac{q^{n}-1}{q^{k}-1}}}$ for $0\le i\le q^{n}-2$ and $0\le j\le\frac{q^{n}-1}{q^{k}-1}-1$. Hence $M_{i}$ has been counted $\frac{q^{n}-1}{q^{k}-1}$ times in Equation (\ref{eq1}). Consequently
\begin{align*}
|\cup_{i,j=0}^{N-1}S_{ij}|\le(\frac{q^{n}-1}{q^{k}-1})^{2}q^{k(n-k)}|\GL(k,q)||\GL(n-k,q)|-\frac{(q^{n}-1)^{2}}{q^{k}-1}+(q^{n}-1).
\end{align*}

By Theorem~\ref{mainthm}, we have
\begin{align*}
P(n,k,q)&\ge\frac{|V(G)|}{|\cup_{i,j=0}^{N-1}S_{ij}|}\\
        &\ge\frac{|\GL(n,q)|}{(\frac{q^{n}-1}{q^{k}-1})^{2}q^{k(n-k)}|\GL(k,q)||\GL(n-k,q)|-\frac{(q^{n}-1)^{2}}{q^{k}-1}+(q^{n}-1)}. \qedhere
\end{align*}
\end{proof}

Since the formulas in both Theorem \ref{thm2} and Theorem \ref{thm3} are complicated, we give the following corollary.
\begin{corollary}\label{coro}
Let $q$ be a prime power and assume that $k|n$. Then
\begin{align*}
P(n,k,q)>\frac{q^{k}-1}{q^{n}-1}{n-1 \brack k-1}_{q}.
\end{align*}
\end{corollary}
\begin{proof}
As $|\GL(n,q)|=(q^n-1)(q^n-q)\cdots(q^n-q^{n-1})$, we get from Theorem \ref{thm3} that
\begin{align*}
P(n,k,q)&\ge\frac{|\GL(n,q)|}{(\frac{q^{n}-1}{q^{k}-1})^{2}q^{k(n-k)}|\GL(k,q)||\GL(n-k,q)|-\frac{(q^{n}-1)^{2}}{q^{k}-1}+(q^{n}-1)}\\
        &>\frac{|\GL(n,q)|}{(\frac{q^{n}-1}{q^{k}-1})^{2}q^{k(n-k)}|\GL(k,q)||\GL(n-k,q)|}\\
        &=\frac{q^{k}-1}{q^{n}-1}{n-1 \brack k-1}_{q}. \qedhere
\end{align*}
\end{proof}
\begin{remark}
The only difference between the lower bound in Corollary~\ref{coro} and the upper bound for partial $k$-parallelisms is the factor $\frac{q^{k}-1}{q^{n}-1}$. It is clear that, our lower bounds are much larger than those given in \cite{B90} and \cite{E15}.
\end{remark}

\section*{Acknowledgment}
The authors express their gratitude to the anonymous reviewers for their detailed and constructive comments which have been very helpful to the improvement of the presentation of this paper. Tao Zhang is supported by the National Natural Science Foundation of China under Grant No.\ 11801109. Yue Zhou is supported by the National Natural Science Foundation of China under Grant No.\ 11771451 and Natural Science Foundation of Hunan Province under Grant No.\ 2019JJ30030.

\end{document}